\documentclass[11pt]{amsart}

\usepackage{a4wide,amsmath,amssymb,dsfont,epsfig,graphicx,subfig,verbatim,multicol,enumitem,lipsum,amsaddr}
\usepackage[hang,flushmargin]{footmisc}

\newtheorem{theorem}{Theorem}
\newtheorem{proposition}[theorem]{Proposition}
\newtheorem{lemma}[theorem]{Lemma}

\newtheorem*{theoremX}{Theorem}

\newtheorem*{stc}{Strong Terminating Conjecture \cite{ShultzShiflett}}
\newtheorem*{wtc}{Weak Terminating Conjecture \cite{ChamberlandMartelli}}
\newtheorem*{cc}{Continuity Conjecture \cite{ChamberlandMartelli}}
\newtheorem*{uc}{Unboundedness Conjecture \cite{HoseanaMSc}}

\newcommand{\mmm}{{\sc mmm}}
\newcommand{\cS}{\mathcal{S}}
\newcommand{\cM}{\mathcal{M}}

\newcommand{\tauA}{\tau_{\textnormal{A}}}
\newcommand{\mA}{m_{\textnormal{A}}}
\newcommand{\bMA}{\mathbf{M}_{\textnormal{A}}}
\newcommand{\rMA}{\textnormal{M}_{\textnormal{A}}}

\usepackage{pgf,tikz}
\usepackage{tkz-fct}
\usepackage{pgfplots}
\usetikzlibrary{arrows,trees}
\pgfplotsset{ every non boxed x axis/.append style={x axis line style=<->},
     every non boxed y axis/.append style={y axis line style=<->}, every axis/.append style={font=\tiny}}
\definecolor{newpurple}{RGB}{195, 22, 140}
\definecolor{newgray}{RGB}{240, 240, 240}
\definecolor{newlightblue}{RGB}{0, 175, 158}
\definecolor{newblue}{RGB}{47, 50, 145}
\definecolor{newyellow}{RGB}{232, 222, 0}
\definecolor{newgreen}{RGB}{0, 155, 1}

\DeclareMathOperator{\median}{median}

\begin{document}
\title{The Akiyama mean-median map has unbounded transit time and discontinuous limit}
\author{\small Jonathan Hoseana}
\address{\normalfont\small Department of Mathematics, Parahyangan Catholic University, Bandung 40141, Indonesia}
\email{j.hoseana@unpar.ac.id}
\date{}

\begin{abstract}
Open conjectures state that, for every $x\in[0,1]$, the orbit $\left(x_n\right)_{n=1}^\infty$ of the mean-median recursion
$$x_{n+1}=(n+1)\cdot\median\left(x_1,\ldots,x_{n}\right)-\left(x_1+\cdots+x_n\right),\quad n\geqslant 3,$$
with initial data $\left(x_1,x_2,x_3\right)=(0,x,1)$, is eventually constant, and that its transit time and limit functions (of $x$) are unbounded and continuous, respectively. In this paper we prove that, for the slightly modified recursion
$$x_{n+1}=n\cdot\median\left(x_1,\ldots,x_{n}\right)-\left(x_1+\cdots+x_n\right),\quad n\geqslant 3,$$
first suggested by Akiyama, the transit time function is unbounded but the limit function is discontinuous.
\end{abstract}

\maketitle

\section{Introduction}\label{section:Introduction}

\sloppy The \textit{mean-median map} (\mmm) enlarges a finite non-empty real set\footnote{The sets on which the \mmm\ acts allow repetitions of elements (i.e., they are \textit{multisets}).} $\left[x_1,\ldots,x_n\right]$ to $\left[x_1,\ldots,x_n,x_{n+1}\right]$, where $x_{n+1}$ is the unique real number which equates the (\textit{arithmetic}) \textit{mean} of the latter set and the \textit{median}\footnote{The middle number after sorting if $n$ is odd, the mean of the middle pair otherwise.} of the former set, namely, 
\begin{equation}\label{eq:originalMMM}
x_{n+1}=(n+1)\cM_n-\cS_n,
\end{equation}
where $\cM_n$ and $\cS_n$ denote the median and the sum of the elements of $\left[x_1,\ldots,x_n\right]$, respectively. Given an initial set $\left[x_1,\ldots,x_{n_0}\right]$, $n_0\in\mathbb{N}$, iterating the map generates an \textit{orbit} $\left(x_n\right)_{n=1}^\infty$ which is conjectured to \textit{stabilise}, i.e., to be eventually constant:\bigskip

\begin{stc}
The \mmm\ orbit of every initial set stabilises.
\end{stc}\bigskip

It is known that the \textit{median sequence} $\left(\cM_n\right)_{n=n_0}^\infty$ associated to the orbit is monotonic \cite[Theorem 2.1]{ChamberlandMartelli}, and converges once a repeated orbit point appears above (below) a median in the non-decreasing (non-increasing) case \cite[Theorem 2.4]{ChamberlandMartelli}. Such repeated points are observed to be ubiquitous \cite[paragraph preceding Section 3]{ChamberlandMartelli}, suggesting:\bigskip

\begin{wtc}
The median sequence of every initial set converges.
\end{wtc}\bigskip


\begin{figure}[t!]
\centering
\input{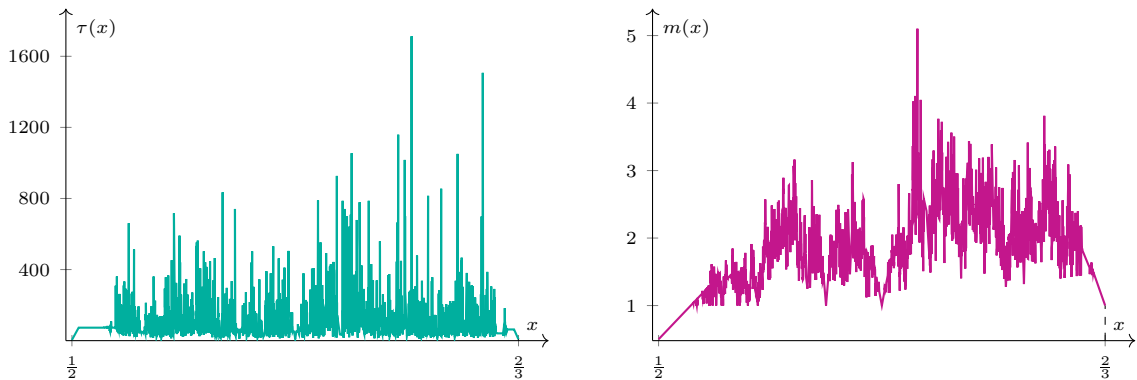}
\caption{\label{fig:mtau}\small
Graphs of $\tau$ (left) and $m$ (right).
}
\end{figure}

Despite intensive research effort \cite{ShultzShiflett,ChamberlandMartelli,HoseanaMSc,CellarosiMunday,HoseanaVivaldi1,HoseanaPhD,HoseanaVivaldi2,Vivaldi}, these terminating conjectures, as well as two additional conjectures to follow, are still open even in the case of smallest non-trivial initial sets: those of size three. The fact that the \mmm\ commutes with elementwise affine transformations \cite[Section 3]{ChamberlandMartelli} makes the orbit of every such set affine-equivalent to that of a univariate initial set $[0,x,1]$, for some real number $x\in\left[\frac{1}{2},\frac{2}{3}\right]$ which we call the \textit{initial condition}. We associate to this set the \textit{transit time} $\tau(x)\in\mathbb{N}_{>3}\cup\{\infty\}$ of its \mmm\ orbit ---the time step at which the orbit stabilises--- and the \textit{limit} $m(x)\in\mathbb{R}$ of its median sequence. These functions, sketched in Figure \ref{fig:mtau}, are conjectured to possess the following properties:\bigskip

\begin{uc}
The function $\tau$ is unbounded.
\end{uc}\medskip

\begin{cc}
The function $m$ is continuous.
\end{cc}\bigskip

A sufficient condition for the appearance of a repeated point ---which guarantees convergence of the median sequence--- is available for bounded rational orbits. Such an orbit is forced to repeat if its time-dependent \textit{effective exponent} ---the largest exponent of $2$ in the denominators of existing points--- grows sublogarithmically over time \cite[equation (2.2)]{HoseanaPhD}. From \eqref{eq:originalMMM} it is apparent that, after each iteration, this exponent either stays unchanged or increases by $1$. Thus, for a sublogarithmic growth, the increments must occur sufficiently infrequently. This infrequency of increments, although well supported by computational evidence, seems to originate from an arithmetical phenomenon which is very difficult to elaborate rigorously.

In order to eliminate this difficulty, Akiyama \cite{Akiyama} suggested modifying the recursion \eqref{eq:originalMMM} into
\begin{equation}\label{eq:akiyamaMMM}
x_{n+1}=n\cM_n-\cS_n,
\end{equation}
thereby introducing a new variant of the \mmm, which we call the \textit{Akiyama} \mmm, whose rational orbits have a \textit{constant} effective exponent. Naturally, for the Akiyama \mmm, there are analogous terminating conjectures; these are also open. However, for this map, clearly, every bounded rational orbit stabilises.

As we shall see, the Akiyama \mmm\ has the same smallest non-trivial form of initial sets, namely $[0,x,1]$, whose transit time $\tauA(x)\in\mathbb{N}_{>3}\cup\{\infty\}$ and limit $\mA(x)\in\mathbb{R}$ are defined analogously for $x\in(-\infty,1)$, and are sketched in Figure \ref{fig:mtauAkiyamaBounds}. For these functions, one naturally questions the analogous Unboundedness and Continuity Conjectures. The main purpose of this paper is to prove analytically that the former holds, whereas the latter fails. More precisely, we will prove:\bigskip

\begin{samepage}
\begin{theoremX}
If $x\in(0,1)$, then $$\tauA(x)\geqslant \frac{2}{x}+3\qquad\qquad\text{and}\qquad\qquad\mA(x)\leqslant 2x-1,$$ where equality holds if and only if $x$ is a unit fraction\footnote{A positive fraction with unit numerator.}.
\end{theoremX}
\end{samepage}\bigskip

\noindent The first inequality clearly implies the unboundedness of $\tauA$. Since $\mA(0)=0$, the second inequality implies that $\mA$ is discontinuous at $x=0$.

Our proof of this theorem is methodologically similar to that of the bounds for the transit time and limit of the so-called \textit{normal form} of the original \mmm\ \cite[Theorem 6.2]{HoseanaVivaldi1}; it goes by first showing that every orbit begins with a \textit{predictable phase} whose length depends on an arithmetical property of the initial condition. The bounds for $\tauA$ and $\mA$ in the theorem then can be inferred from, respectively, the number of existing points and the location of the median at the end of the phase.

The simultaneous occurrence of the unboundedness of the transit time and the discontinuity of the limit function is unsurprising. Indeed, in the original \mmm\ we have pointed out that these will be two interrelated consequences if a local functional orbit is found to be divergent \cite[Theorems 5.4 and 5.6]{HoseanaVivaldi1}. While such divergence has not been found in the original \mmm, we find it near $x=0$ in the Akiyama \mmm.

Let us now describe the structure of this paper. In the upcoming section we define the Akiyama \mmm\ more formally and discuss its basic properties. There are properties which are the same as those of the original \mmm\ (the proofs of which are thus omitted): the median sequence is monotonic (Proposition \ref{prop:monotonicmedian}), a repeated orbit point guarantees convergence and two equal consecutive medians cause stabilisation (Proposition \ref{prop:prop1}), as well as a different one: the map commutes with scalar multiplications, but not with non-identity translations (Proposition \ref{prop:commute}). In Section \ref{sec:Main result} we present our main result, namely an explicit description of the predictable phase for every initial condition (Lemma \ref{lemma:regphase}) from which the above theorem is then proved to follow. Finally, the graphs in Figure \ref{fig:mtauAkiyamaBounds} suggest the presence of symmetry around $x=\frac{1}{2}$; a brief discussion on this in Section \ref{sec:Symmetries} concludes the paper.


\begin{figure}[t!]
\centering
\input{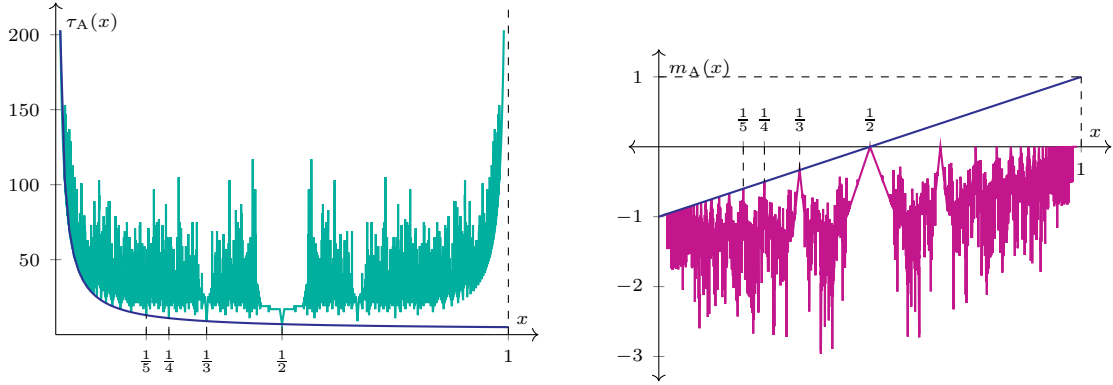}
\caption{\label{fig:mtauAkiyamaBounds}\small
Graphs of $\tauA$ (left) and $\mA$ (right) in $(0,1)$ with bounds given in the main theorem.
}
\end{figure}

\section{Preliminaries}\label{sec:Preliminaries}

The Akiyama \mmm\ is a self-map on the space of finite non-empty real sets. The image $\bMA(\xi)$ of such a set $\xi$ is obtained by increasing the multiplicity of the real number
$$\rMA(\xi):=|\xi|\cM(\xi)-\cS(\xi)$$
in $\xi$ by one, where $|\xi|$, $\cM(\xi)$, and $\cS(\xi)$ denote the cardinality, median, and sum of elements of $\xi$, respectively. Employing the additive union notation \cite[page 50]{Blizard}, we write
$$\bMA(\xi):=\xi\uplus[\rMA(\xi)].$$
Generally, the map $\bMA$ does \textit{not} commute with elementwise affine transformations (cf.~\cite[Theorem 2.2]{ChamberlandMartelli}). However, it commutes with elementwise scalar multiplications:\bigskip

\begin{proposition}\label{prop:commute}\textcolor{white}{a}
For every $a,b\in\mathbb{R}$ with $a\neq0$ we have
$$\bMA(a\xi+b)=(a\xi+b)\uplus[a\rMA(\xi)],$$
and, in particular,
\begin{equation}\label{eq:scalarmultiplication}
\bMA(a\xi)=a\bMA(\xi),
\end{equation}
i.e., $\bMA$ commutes with elementwise scalar multiplications.
\end{proposition}
\begin{proof}
Since $\cM(a\xi+b)=a\cM(\xi)+b$ and $\cS(a\xi+b)=a\cS(\xi)+|\xi|b$, the map $\bMA$ increases in the set $a\xi+b$ the multiplicity of the number
\begin{eqnarray*}
|a\xi+b|\cM(a\xi+b)-\cS(a\xi+b)&=&|\xi|\left[a\cM(\xi)+b\right]-\left[a\cS(\xi)+|\xi|b\right]\\
                               &=&a\left[|\xi|\cM(\xi)-\cS(\xi)\right]\\
                               &=&a\rMA(\xi),
\end{eqnarray*}
proving the first identity. Setting $b=0$ gives the second identity.
\end{proof}\bigskip

Under iterations of $\bMA$, every initial set $\xi_{n_0}=\left[x_1,\ldots,x_{n_0}\right]$, $n_0\in\mathbb{N}$, is associated to a sequence of sets $\left(\xi_n\right)_{n=n_0}^\infty$, an orbit $\left(x_n\right)_{n=1}^\infty$, and a median sequence $\left(\cM_n\right)_{n=n_0}^\infty$, where
$$\xi_{n+1}=\bMA\left(\xi_{n}\right),\quad x_{n+1}=\rMA\left(\xi_n\right),\quad\text{and}\quad \cM_n:=\cM\left(\xi_n\right),\quad\text{for every }n\geqslant n_0.$$
Moreover, we have
\begin{equation}\label{eq:affcombmedians}
x_{n+2}=(n+1)\cM_{n+1}-n\cM_n,\qquad\text{for every }n\geqslant n_0,
\end{equation}
an expression of an orbit point as an affine combination of the last two medians. Exactly as in the original \mmm\ \cite[Theorem 2.1]{ChamberlandMartelli}, we deduce from \eqref{eq:affcombmedians} that the median sequence is monotonic:\bigskip

\begin{proposition}\label{prop:monotonicmedian}
The median sequence $\left(\cM_n\right)_{n=n_0}^\infty$ is monotonic.
\end{proposition}\bigskip

Loosely speaking, an Akiyama \mmm\ orbit reaches stabilisation in a similar way as an original \mmm\ orbit: the orbit first generates a repeated point which guarantees the convergence of the median sequence\footnote{In the case of $x_1,\ldots,x_{n_0}\in\mathbb{Q}$, since the effective exponent is constant, convergence implies stabilisation.} \cite[Theorem 2.4]{ChamberlandMartelli}. Once one of these repeated points is reached by the median sequence, two equal consecutive medians are created; as apparent from \eqref{eq:affcombmedians}, this causes stabilisation. Formally, we have:\bigskip

\begin{proposition}\label{prop:prop1}\textcolor{white}{a}
\begin{enumerate}
\item[\textnormal{(i)}] If $n\geqslant n_0$ is such that $\cM_n=\cM_{n+1}$, then $x_j=\cM_{n+1}$ for every $j\geqslant n+2$.
\item[\textnormal{(ii)}] The non-decreasing (non-increasing) median sequence converges if there exist $i,j,s\in\mathbb{N}$ with $i\neq j$ and $s\geqslant n_0$ such that $\cM_s\leqslant x_i=x_j$ ($\cM_s\geqslant x_i=x_j$).
\end{enumerate}
\end{proposition}\bigskip

The orbits of a singleton set $[x]$, a two-element set containing a zero $[0,x]$, and a set of two equal elements $[x,x]$, where $x\in\mathbb{R}$, are straightforward to compute; these are $(x,0,0,-x,\overline{0})$, $(0,x,0,-x,\overline{0})$, and $(x,x,0,\overline{x})$, respectively. The smallest non-trivial initial sets are those of the form $[x,y]$, where $x$, $y$ are non-zero and $x<y$. By \eqref{eq:scalarmultiplication}, these are represented by sets of the form $[x,1]$, $x<1$, whose limit $\mA(x)$ and transit time $\tauA(x)$ are plotted in Figure \ref{fig:mtauAkiyamaBounds}. For these sets the median sequence is non-increasing. It is straightforward to show that $\bMA([x,1])=[0,x,1]$; in this sense the smallest non-trivial initial sets of the original and Akiyama \mmm s have the same form.

\section{Main result}\label{sec:Main result}

We are now ready to present our main result. For $x\in(0,1)$, we show that the orbit of the smallest non-trivial initial set $[x,1]$ begins with a \textit{predictable phase}: an initial segment of length $2\ell+2$, where $\ell:=\left\lceil\frac{1}{x}\right\rceil\geqslant 2$, in which every term has an explicit formula. In this phase, the first four terms are given by $\left(x_n\right)_{n=1}^4=(x,1,0,2x-1)$, as easily verified, and the rest by the following lemma. Moreover, the phase is followed by stabilisation ---hence the available formulae describe the entire orbit--- if and only if $x$ is a unit fraction, i.e., the reciprocal of $\ell$. See Figure \ref{fig:regphase}.\bigskip

\begin{lemma}\label{lemma:regphase}
Let $x_n$ be the $n$-th term of the orbit of the set $[x,1]$, where $x\in(0,1)$.
\begin{enumerate}
\item[\textnormal{(i)}] If $x=\frac{1}{\ell}$ for some integer $\ell\geqslant 2$, then $x_n=-(n-4)x$ for every $n\in\{5,\ldots,2\ell+2\}$, and $x_n=2x-1$ for every $n\geqslant 2\ell+3$. Thus, $\mA(x)=2x-1$ and $\tauA(x)=2\ell+3$.
\item[\textnormal{(ii)}] If $x\in\left(\frac{1}{\ell},\frac{1}{\ell-1}\right)$ for some integer $\ell\geqslant 2$, then $x_n=-(n-4)x$ for every $n\in\{5,\ldots,2\ell\}$,
\begin{equation}\label{eq:regphaselemma}
x_{2\ell+1}=\left(\ell^2-2\ell+3\right)x-\ell,\quad\text{and}\quad x_{2\ell+2}=\left(\ell^2-\ell+2\right)x-\ell-1.
\end{equation}
Moreover, $\mA(x)<2x-1$ and $\tauA(x)>2\ell+3$.
\end{enumerate}
\end{lemma}
\begin{proof}
Let $x\in\left[\frac{1}{\ell},\frac{1}{\ell-1}\right)$ for some integer $\ell\geqslant 2$. First, suppose $\ell=2$. Then $x\in\left[\frac{1}{2},1\right)$. If $x=\frac{1}{2}$, then $\left(x_n\right)_{n=1}^\infty=\left(\frac{1}{2},1,0,0,-\frac{1}{2},-1,\overline{0}\right)$, satisfying (i). Otherwise, $\left(x_n\right)_{n=1}^6=(x,1,0,2x-1,3x-2,4x-3)$, satisfying (ii).

Therefore, it remains to prove the lemma for $\ell\geqslant 3$. In this case, we have $x\in\left(0,\frac{1}{2}\right)$. We divide the proof into two parts.\bigskip

\noindent\underline{\textsc{Part I}: Formulae for $x_5$, \ldots, $x_{2\ell}$}. Let us prove that for every $n\in\{5,\ldots,2\ell\}$ we have
\begin{equation}\label{eq:regphase0}
x_n=-(n-4)x
\end{equation}
by strong induction on $n$.
First, since $x\in\left(0,\frac{1}{2}\right)$, then $x_4<x_3<x_1<x_2$, so $\cM_4=\left\langle x_3,x_1\right\rangle=\frac{x}{2}$ and
$$x_5=4\cM_4-\cS_4=4\cdot\frac{x}{2}-3x=-x,$$
proving that the statement holds for $n=5$.

Next, let $r\in\{5,\ldots,2\ell-1\}$ be such that $x_n=-(n-4)x$ for every $n\in\{5,\ldots,r\}$. We shall prove that $x_{r+1}=-(r-3)x$, dividing the proof into two cases:\bigskip

\noindent\textsc{Case I}: $r\in\{5,\ldots,\ell+1\}$. Since $x<\frac{1}{\ell-1}$, then
$$x_4-x_{r}=(2x-1)+(r-4)x\leqslant (2x-1)+[(\ell+1)-4]x<0,\qquad\text{i.e.,}\qquad x_4<x_r,$$
so
$$x_4<x_r<x_{r-1}<\cdots<x_5<x_3<x_1<x_2,$$
from which we can see that, if $r$ is odd,
$$\cM_{r-1}=\begin{cases}
\left\langle x_{r-2},x_{r-4}\right\rangle,&\text{if }r\in\{5,7\};\\
\left\langle x_{\frac{r+3}{2}}, x_{\frac{r+1}{2}}\right\rangle,&\text{if }r\geqslant 9
\end{cases}\qquad\text{and}\qquad\cM_{r}=\begin{cases}
x_3,&\text{if }r=5;\\
x_{\frac{r+3}{2}},&\text{if }r\geqslant 7,
\end{cases}$$
otherwise
$$\cM_{r-1}=\begin{cases}
x_3,&\text{if }r=6;\\
x_{\frac{r+2}{2}},&\text{if }r\geqslant 8
\end{cases}\qquad\text{and}\qquad\cM_{r}=\begin{cases}
\left\langle x_{3},x_{5}\right\rangle,&\text{if }r=6;\\
\left\langle x_{\frac{r+4}{2}}, x_{\frac{r+2}{2}}\right\rangle,&\text{if }r\geqslant 8.
\end{cases}\medskip$$

\noindent\textsc{Case II}: $r\in\{\ell+2,\ldots,2\ell-1\}$. Since $\frac{1}{\ell}\leqslant x<\frac{1}{\ell-1}$, then
$$x_4-x_{\ell+1}=(2x-1)+[(\ell+1)-4]x<0,\qquad\text{i.e.,}\qquad x_4<x_{\ell+1}$$
and
$$x_{\ell+2}-x_4=[(\ell+2)-4]x-(2x-1)\leqslant 0,\qquad\text{i.e.,}\qquad x_{\ell+2}\leqslant x_4,$$
so
$$x_{r}<\cdots<x_{\ell+2}\leqslant x_4<x_{\ell+1}<x_{\ell}<\cdots<x_5<x_3<x_1<x_2,$$
from which we can see that
$$\cM_{r-1}=\begin{cases}
\left\langle x_{\frac{r+3}{2}}, x_{\frac{r+1}{2}}\right\rangle,&\text{if }r\text{ is odd};\\
x_{\frac{r+2}{2}},&\text{otherwise}
\end{cases}\qquad\text{and}\qquad\cM_{r}=\begin{cases}
x_{\frac{r+3}{2}},&\text{if }r\text{ is odd};\\
\left\langle x_{\frac{r+4}{2}}, x_{\frac{r+2}{2}}\right\rangle,&\text{otherwise}.
\end{cases}$$\bigskip

In both cases we have $\cM_{r-1}=-\frac{r-6}{2}x$ and $\cM_{r}=-\frac{r-5}{2}x$, so
$$x_{r+1}=r\cM_{r}-(r-1)\cM_{r-1}=r\left(-\frac{r-5}{2}x\right)-(r-1)\left(-\frac{r-6}{2}x\right)=-(r-3)x,$$
as desired.\bigskip

\noindent\underline{\textsc{Part II}: Formulae for $x_{2\ell+1}$ and $x_{2\ell+2}$}. From the previous part we know that
$\cM_{2\ell-1}=x_{\ell+1}$. Moreover, since
$$x_{2\ell}<\cdots<x_{\ell+2}\leqslant x_4<x_{\ell+1}<x_{\ell}<\cdots<x_5<x_3<x_1<x_2,$$
then $\cM_{2\ell}=\left\langle x_4, x_{\ell+1}\right\rangle$. Therefore,
\begin{equation}\label{eq:regphase1}
x_{2\ell+1}=2\ell\cM_{2\ell}-(2\ell-1)\cM_{2\ell-1}=\ell x_4-(\ell-1)x_{\ell+1}=x_4-(\ell-1)\left(x_{\ell+1}-x_4\right)<x_4,
\end{equation}
so that $\cM_{2\ell+1}=x_4$, implying
\begin{equation}\label{eq:regphase2}
x_{2\ell+2}=(2\ell+1)\cM_{2\ell+1}-2\ell\cM_{2\ell}=(\ell+1)x_4-\ell x_{\ell+1}.
\end{equation}
Next, we split into two cases:\bigskip

\noindent\textsc{Case I}: $x=\frac{1}{\ell}$. In this case, $x_4=x_{\ell+2}=-(\ell-2)x$. Substituting this and $x_{\ell+1}=-(\ell-3)x$ into \eqref{eq:regphase1} and \eqref{eq:regphase2} gives $x_{2\ell+1}=-(2\ell-3)x$ and $x_{2\ell+2}=-(2\ell-2)x$, extending the formula \eqref{eq:regphase0}. Moreover, since $x_{2\ell+2}<x_{2\ell+1}<x_4$, then $\cM_{2\ell+2}=\left\langle x_{\ell+2},x_4\right\rangle=x_4=\cM_{2\ell+1}$, so, by part (ii) of Proposition \ref{prop:prop1}, we have $x_n=x_4=2x-1$ for every $n\geqslant 2\ell+3$. This means $\mA(x)=2x-1$ and $\tauA(x)=2\ell+3$, completing the proof.\medskip

\noindent\textsc{Case II}: $x\in\left(\frac{1}{\ell},\frac{1}{\ell-1}\right)$. Substituting $x_4=2x-1$ and $x_{\ell+1}=-(\ell-3)x$ into \eqref{eq:regphase1} and \eqref{eq:regphase2} gives \eqref{eq:regphaselemma}. Moreover, we have
$$x_{2\ell+2}=\left(\ell^2-\ell+2\right)x-\ell-1=2x-1+\ell(\ell-1)x-\ell<2x-1=x_4,$$
because $\ell(\ell-1)x-\ell<0$ as $x<\frac{1}{\ell-1}$. Consequently, $\cM_{2\ell+2}<\mathcal{M}_{2\ell+1}$, so $\mA(x)<\mathcal{M}_{2\ell+1}=2x-1$ and $\tauA(x)>2\ell+3$, completing the proof.
\end{proof}\bigskip

\begin{figure}
\centering
\begin{tabular}{ccc}
\begin{tikzpicture}
\begin{axis}[
	xmin=0,
	xmax=26.66666667,
	ymin=-2.099800000,
	ymax=1.281800000,
    xtick={1,5,25},
    ytick={-.8182,0.9091e-1,1},
    yticklabels={$-\frac{9}{11}$,$\frac{1}{11}$,$1$},
	axis lines=middle,
    x axis line style=->,
    y axis line style=<->,
	samples=100,
	xlabel=$n$,
	ylabel=$x_n$,
	width=8cm,
    height=6cm
]
\draw[dashed] (axis cs:0,1) -- (axis cs:2,1);
\draw[dashed] (axis cs:1,0) -- (axis cs:1,0.9091e-1);
\draw[dashed] (axis cs:0,0.9091e-1) -- (axis cs:1,0.9091e-1);
\draw[dashed] (axis cs:25,0) -- (axis cs:25,-.8182);
\draw[dashed] (axis cs:25, -.8182) -- (axis cs:0,-.8182);
\draw[dashed] (axis cs:5,0) -- (axis cs:5, -0.9091e-1);

\draw[thick,color=newblue] plot coordinates {(axis cs:1, 0.9091e-1) (axis cs:2, 1.) (axis cs:3, 0.) (axis cs:4, -.8182) (axis cs:5, -0.9091e-1) };

\draw[thick,color=newlightblue] plot coordinates {(axis cs:5, -0.9091e-1) (axis cs:6, -.1818) (axis cs:7, -.2727) (axis cs:8, -.3636) (axis cs:9, -.4545) (axis cs:10, -.5455) (axis cs:11, -.6364) (axis cs:12, -.7273) (axis cs:13, -.8182) (axis cs:14, -.9091) (axis cs:15, -1.) (axis cs:16, -1.091) (axis cs:17, -1.182) (axis cs:18, -1.273) (axis cs:19, -1.364) (axis cs:20, -1.455) (axis cs:21, -1.545) (axis cs:22, -1.636) (axis cs:23, -1.727) (axis cs:24, -1.818) (axis cs:25, -.8182)};

\draw[thick,color=newgreen] plot coordinates {(axis cs:24, -1.818) (axis cs:25, -.8182)};

\fill[newblue] (axis cs:1, 0.9091e-1) circle(2pt);
\fill[newblue] (axis cs:2, 1.) circle(2pt);
\fill[newblue] (axis cs:3, 0.) circle(2pt);
\fill[newblue] (axis cs:4, -.8182) circle(2pt);
\fill[newlightblue] (axis cs:5, -0.9091e-1) circle(2pt);
\fill[newlightblue] (axis cs:6, -.1818) circle(2pt);
\fill[newlightblue] (axis cs:7, -.2727) circle(2pt);
\fill[newlightblue] (axis cs:8, -.3636) circle(2pt);
\fill[newlightblue] (axis cs:9, -.4545) circle(2pt);
\fill[newlightblue] (axis cs:10, -.5455) circle(2pt);
\fill[newlightblue] (axis cs:11, -.6364) circle(2pt);
\fill[newlightblue] (axis cs:12, -.7273) circle(2pt);
\fill[newlightblue] (axis cs:13, -.8182) circle(2pt);
\fill[newlightblue] (axis cs:14, -.9091) circle(2pt);
\fill[newlightblue] (axis cs:15, -1.) circle(2pt);
\fill[newlightblue] (axis cs:16, -1.091) circle(2pt);
\fill[newlightblue] (axis cs:17, -1.182) circle(2pt);
\fill[newlightblue] (axis cs:18, -1.273) circle(2pt);
\fill[newlightblue] (axis cs:19, -1.364) circle(2pt);
\fill[newlightblue] (axis cs:20, -1.455) circle(2pt);
\fill[newlightblue] (axis cs:21, -1.545) circle(2pt);
\fill[newlightblue] (axis cs:22, -1.636) circle(2pt);
\fill[newlightblue] (axis cs:23, -1.727) circle(2pt);
\fill[newlightblue] (axis cs:24, -1.818) circle(2pt);
\fill[newgreen] (axis cs:25, -.8182) circle(2pt);
\end{axis}
\end{tikzpicture}&\phantom{a}&\begin{tikzpicture}
\begin{axis}[
	xmin=0,
	xmax=41.60000000,
	ymin=-3.138200000,
	ymax=1.376200000,
    xtick={1,5,22,39},
    ytick={-1.714,-1.333,0.9524e-1,1},
    yticklabels={$-\frac{12}{7}$,$-\frac{4}{3}$,$\frac{2}{21}$,$1$},
	axis lines=middle,
    x axis line style=->,
    y axis line style=<->,
	samples=100,
	xlabel=$n$,
	ylabel=$x_n$,
	width=8cm,
    height=6cm
]
\draw[dashed] (axis cs:0,1) -- (axis cs:2,1);
\draw[dashed] (axis cs:1,0) -- (axis cs:1,0.9524e-1);
\draw[dashed] (axis cs:0,0.9524e-1) -- (axis cs:1,0.9524e-1);
\draw[dashed] (axis cs:22,0) -- (axis cs:22, -1.714);
\draw[dashed] (axis cs:0, -1.714) -- (axis cs:22, -1.714);
\draw[dashed] (axis cs:0,-1.333) -- (axis cs:39, -1.333);
\draw[dashed] (axis cs:39, 0) -- (axis cs:39, -1.333);
\draw[dashed] (axis cs:5,0) -- (axis cs:5, -0.9524e-1);

\draw[thick,color=newblue] plot coordinates {(axis cs:1, 0.9524e-1) (axis cs:2, 1.) (axis cs:3, 0.) (axis cs:4, -.8095) (axis cs:5, -0.9524e-1)};

\draw[thick,color=newlightblue] plot coordinates {(axis cs:5, -0.9524e-1) (axis cs:6, -.1905) (axis cs:7, -.2857) (axis cs:8, -.3810) (axis cs:9, -.4762) (axis cs:10, -.5714) (axis cs:11, -.6667) (axis cs:12, -.7619) (axis cs:13, -.8571) (axis cs:14, -.9524) (axis cs:15, -1.048) (axis cs:16, -1.143) (axis cs:17, -1.238) (axis cs:18, -1.333) (axis cs:19, -1.429) (axis cs:20, -1.524) (axis cs:21, -1.619) (axis cs:22, -1.714)};

\draw[thick,color=newpurple] plot coordinates {(axis cs:22, -1.714) (axis cs:23, -1.286) (axis cs:24, -1.333) (axis cs:25, -1.381) (axis cs:26, -1.429) (axis cs:27, -2.095) (axis cs:28, -2.190) (axis cs:29, -2.286) (axis cs:30, -2.381) (axis cs:31, -2.476) (axis cs:32, -2.571) (axis cs:33, -2.667) (axis cs:34, -2.762) (axis cs:35, -2.048) (axis cs:36, -2.095) (axis cs:37, -2.143) (axis cs:38, -2.190)};

\draw[thick,color=newgreen] plot coordinates {(axis cs:38, -2.190) (axis cs:39, -1.333)};

\fill[newblue] (axis cs:1, 0.9524e-1) circle(2pt);
\fill[newblue] (axis cs:2, 1.) circle(2pt);
\fill[newblue] (axis cs:3, 0.) circle(2pt);
\fill[newblue] (axis cs:4, -.8095) circle(2pt);
\fill[newlightblue] (axis cs:5, -0.9524e-1) circle(2pt);
\fill[newlightblue] (axis cs:6, -.1905) circle(2pt);
\fill[newlightblue] (axis cs:7, -.2857) circle(2pt);
\fill[newlightblue] (axis cs:8, -.3810) circle(2pt);
\fill[newlightblue] (axis cs:9, -.4762) circle(2pt);
\fill[newlightblue] (axis cs:10, -.5714) circle(2pt);
\fill[newlightblue] (axis cs:11, -.6667) circle(2pt);
\fill[newlightblue] (axis cs:12, -.7619) circle(2pt);
\fill[newlightblue] (axis cs:13, -.8571) circle(2pt);
\fill[newlightblue] (axis cs:14, -.9524) circle(2pt);
\fill[newlightblue] (axis cs:15, -1.048) circle(2pt);
\fill[newlightblue] (axis cs:16, -1.143) circle(2pt);
\fill[newlightblue] (axis cs:17, -1.238) circle(2pt);
\fill[newlightblue] (axis cs:18, -1.333) circle(2pt);
\fill[newlightblue] (axis cs:19, -1.429) circle(2pt);
\fill[newlightblue] (axis cs:20, -1.524) circle(2pt);
\fill[newlightblue] (axis cs:21, -1.619) circle(2pt);
\fill[newlightblue] (axis cs:22, -1.714) circle(2pt);
\fill[newpurple] (axis cs:23, -1.286) circle(2pt);
\fill[newpurple] (axis cs:24, -1.333) circle(2pt);
\fill[newpurple] (axis cs:25, -1.381) circle(2pt);
\fill[newpurple] (axis cs:26, -1.429) circle(2pt);
\fill[newpurple] (axis cs:27, -2.095) circle(2pt);
\fill[newpurple] (axis cs:28, -2.190) circle(2pt);
\fill[newpurple] (axis cs:29, -2.286) circle(2pt);
\fill[newpurple] (axis cs:30, -2.381) circle(2pt);
\fill[newpurple] (axis cs:31, -2.476) circle(2pt);
\fill[newpurple] (axis cs:32, -2.571) circle(2pt);
\fill[newpurple] (axis cs:33, -2.667) circle(2pt);
\fill[newpurple] (axis cs:34, -2.762) circle(2pt);
\fill[newpurple] (axis cs:35, -2.048) circle(2pt);
\fill[newpurple] (axis cs:36, -2.095) circle(2pt);
\fill[newpurple] (axis cs:37, -2.143) circle(2pt);
\fill[newpurple] (axis cs:38, -2.190) circle(2pt);
\fill[newgreen] (axis cs:39, -1.333) circle(2pt);
\end{axis}
\end{tikzpicture}
\end{tabular}
\caption{\label{fig:regphase}\small
The orbit of $[x,1]$ for $x=\frac{1}{11}$ (left) and for $x=\frac{2}{21}\in\left(\frac{1}{11},\frac{1}{10}\right)$ (right). The first four terms are shown in dark blue, the terms prescribed by Lemma \ref{lemma:regphase} in light blue, the unprescribed terms in purple, and the term from which the orbit stabilises in green.
}
\end{figure}
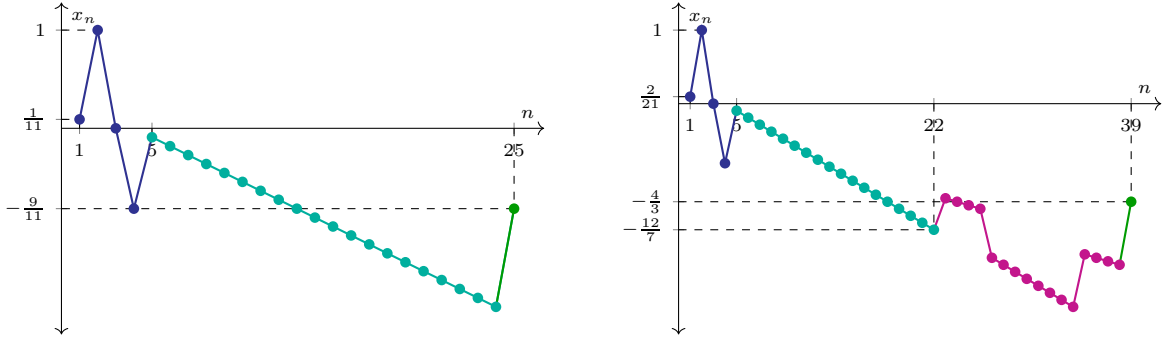

To show how our main theorem follows from Lemma \ref{lemma:regphase}, let $x\in(0,1)$. If $x=\frac{1}{\ell}$ for some integer $\ell\geqslant 2$, then, by Lemma \ref{lemma:regphase}, we have $\mA(x)=2x-1$ and $\tauA(x)=2\ell+3=\frac{2}{x}+3$. Otherwise, $x\in\left(\frac{1}{\ell},\frac{1}{\ell-1}\right)$ for some integer $\ell\geqslant 2$, so by Lemma \ref{lemma:regphase}, $\mA(x)<2x-1$ and $\tauA(x)\geqslant2\ell+3=\frac{2}{\frac{1}{\ell}}+3>\frac{2}{x}+3$.

\section{Remarks on symmetries}\label{sec:Symmetries}

One of the most striking features of Figure \ref{fig:mtauAkiyamaBounds} is the presence of symmetries, particularly around $x=\frac{1}{2}$. In this closing section, we briefly explain the symmetry near $x=\frac{1}{2}$ in the light of what has been done for the original \mmm\ \cite{HoseanaVivaldi1}.

As in \cite{HoseanaVivaldi1}, we now regard $[x,1]$, $x\in(0,1)$, as a set of univariate piecewise-affine continuous real functions [in this case $Y_1(x)=x$ and $Y_2(x)=1$]; we refer to such a set as a \textit{bundle} \cite[Section 2.2]{HoseanaVivaldi1}. Observing that
$$\bMA([x,1])=[x,1,0]\quad\qquad\text{and}\qquad\quad\bMA([x,1,0])=[x,1,0,2x-1],$$ it is natural to regard $\bMA$ as a self-map on the space of non-empty bundles with pointwise action.

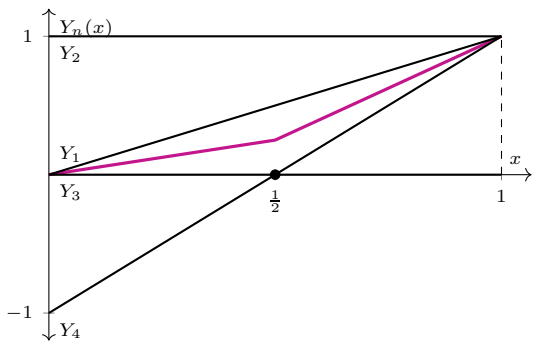
\begin{figure}
\centering
\begin{tikzpicture}
\begin{axis}[
	xmin=0,
	xmax=1.066666667,
	ymin=-1.200000000,
	ymax=1.200000000, 
    xtick={0.5,1},
    ytick={-1,1},
    xticklabels={$\frac{1}{2}$,$1$},
	axis lines=middle,
	samples=100,
	xlabel=$x$,
	ylabel=$Y_n(x)$,
	width=8cm,
	height=6cm,
	clip=false,
    x axis line style=->,
    y axis line style=<->,
]
\draw[newpurple,very thick] (axis cs:0,0) -- (axis cs:0.5,0.25) -- (axis cs:1,1);

\addplot[thick,domain=0:1] {0};
\addplot[thick,domain=0:1] {x};
\addplot[thick,domain=0:1] {1};
\addplot[thick,domain=0:1] {2*x-1};

\fill[black] (axis cs:0.5,0) circle (2pt);

\draw[dashed] (axis cs:1,0) -- (axis cs:1,1);

\node[above right,yshift=1pt] at (axis cs:0,0) {$Y_1$};
\node[below right] at (axis cs:0,1) {$Y_2$};
\node[below right] at (axis cs:0,0) {$Y_3$};
\node[below right] at (axis cs:0,-1) {$Y_4$};
\end{axis}
\end{tikzpicture}
\caption{\label{fig:symmetry}\small
The bundle $[x,1,0,2x-1]$ and its median $\cM_4$ in purple.
}
\end{figure}

The point $\frac{1}{2}$ is an \textit{X-point} \cite[Section 2.2]{HoseanaVivaldi1}: a transversal intersection of two bundle functions, namely $Y_3(x)=0$ and $Y_4(x)=2x-1$ (see Figure \ref{fig:symmetry}). Let $$\Omega:=\left[Y_3,Y_4,Y_1\right]$$ be the subbundle containing these two functions and the function $Y_1$ immediately above the X-point. Notice that, for
\begin{equation}\label{eq:muf}
f(z)=z-2x+1\qquad\qquad\text{and}\qquad\qquad \mu(x)=1-x,
\end{equation}
the subbundle $\Omega$ satisfies
$$\Omega(\mu(x))=[\mu(x),2\mu(x)-1,0]=[1-x,-2x+1,0]=f([0,2x-1,x])=f(\Omega(x)).$$
Moreover, it is possible to show that the set of all functions $Y$ satisfying the same identity, $Y(\mu(x))=f(Y(x))$, is precisely
$$\Psi:=\bigl\{\alpha\min\left\{Y_3,Y_4\right\}+\beta\max\left\{Y_3,Y_4\right\}+\gamma Y_1:\alpha+\beta+\gamma=1\bigr\},$$
i.e., the set of all affine combinations of the functions $\min\left\{Y_3,Y_4\right\}$, $\max\left\{Y_3,Y_4\right\}$, and $Y_1$, the minimum and maximum being defined pointwise \cite[Lemma 5.1]{HoseanaVivaldi1}.

One shows that
$$Y_5=4\cM_4-3\cM_3=0\cdot\min\left\{Y_3,Y_4\right\}+2\cdot\max\left\{Y_3,Y_4\right\}+(-1)\cdot Y_1\in\Psi.$$
Moreover, for every $n\geqslant 5$, the fact that $Y_5,\ldots,Y_n\in\Psi$ implies $Y_{n+1}\in\Psi$, since
$$Y_{n+1}=n\cM_n-(n-1)\cM_{n-1}$$
is an affine combination of $\cM_n$ and $\cM_{n-1}$, each of which is either a function in the set $\left[Y_5,\ldots,Y_n\right]\uplus\left[\min\left\{Y_3,Y_4\right\},\max\left\{Y_3,Y_4\right\},Y_1\right]$ or the mean of two such functions. This inductively proves that $Y_n\in\Psi$ for every $n\geqslant 5$ (cf.~\cite[Lemma 5.2]{HoseanaVivaldi1}).

In other words, we have
$$Y_n(\mu(x))=f\left(Y_n(x)\right)$$
for every $n\geqslant 5$, where $f$ and $\mu$ are given by \eqref{eq:muf}. Since $\mu:\left(0,\frac{1}{2}\right]\to\left[\frac{1}{2},1\right)$ is a bijection, the transformation $f$ connects the dynamics at every initial condition $x\in\left(0,\frac{1}{2}\right]$ to that at a unique initial condition $\mu(x)\in\left[\frac{1}{2},1\right)$. In particular, for every $x\in\left(0,\frac{1}{2}\right]$, we have
$$\mA(\mu(x))=f\left(\mA(x)\right)\qquad\qquad\text{and}\qquad\qquad\tauA(\mu(x))=\tauA(x),$$
i.e.,
$$\mA(1-x)=\mA(x)-2x+1\qquad\qquad\text{and}\qquad\qquad\tauA(1-x)=\tauA(x),$$
explaining the symmetry seen in Figure \ref{fig:mtauAkiyamaBounds}.

The symmetry also means that the bounds in our main theorem ---although already sufficient to achieve the goal of this paper--- can be improved as
$$\mA(x)\leqslant\begin{cases}
2x-1,&\text{if }x\in\left(0,\frac{1}{2}\right];\\
0,&\text{if }x\in\left(\frac{1}{2},1\right)
\end{cases}\qquad\text{and}\qquad
\tauA(x)\geqslant\begin{cases}
\frac{2}{x}+3,&\text{if }x\in\left(0,\frac{1}{2}\right];\\
\frac{2}{1-x}+3,&\text{if }x\in\left(\frac{1}{2},1\right),
\end{cases}$$
where equalities in $\left(0,\frac{1}{2}\right]$ occur at unit fractions, whereas those in $\left[\frac{1}{2},1\right)$ occur at fractions whose numerator and denominator differ by $1$. These two families of fractions form two sequences, converging to the points $0$ and $1$ where $\mA$ is discontinuous, along which $\tauA$ becomes arbitrarily large.

\section*{Acknowledgments} The author thanks Shigeki Akiyama, who first suggested this variant of the \mmm; Franco Vivaldi, through whom the suggestion was communicated; and MATRIX, the organiser of the conference which made the communication possible \cite{Vivaldi}.


\end{document}